\documentclass[smallextended]{my_svjour3}   

\smartqed  

\usepackage{amssymb,amsmath,epsfig,multicol,graphicx,epstopdf}
\usepackage{url}
\usepackage{a4wide}

\begin{document}

\title{Numerical solution of a class of third-kind Volterra integral equations using Jacobi wavelets}

\titlerunning{Numerical solution of a class of third-kind Volterra integral equations using Jacobi wavelets}

\author{S. Nemati \and Pedro M. Lima \and Delfim F. M. Torres}

\institute{S. Nemati 
\at Department of Mathematics, Faculty of Mathematical Sciences, 
University of Mazandaran, Babolsar, Iran\\
\email{s.nemati@umz.ac.ir}
\and Pedro M. Lima 
\at Centro de Matem\'{a}tica Computacional e Estoc\'astica, 
Instituto Superior T\'{e}cnico, Universidade de Lisboa,\\ 
Av. Rovisco Pais, 1049-001 Lisboa, Portugal\\
\email{pedro.t.lima@ist.utl.pt}
\and Delfim F. M. Torres 
\at Center for Research and Development in Mathematics and Applications (CIDMA),
Department of Mathematics, University of Aveiro, 3810-193 Aveiro, Portugal\\
\email{delfim@ua.pt}}

\date{Submitted: 12-Oct-2019 / Revised: 17-Dec-2019 / Accepted: 11-Feb-2020}

% ------------------------------------------

\maketitle

\begin{abstract}
We propose a spectral collocation method, based on the generalized 
Jacobi wavelets along with the Gauss--Jacobi quadrature formula, 
for solving a class of third-kind Volterra integral equations. 
To do this, the interval of integration is first transformed into the interval $[-1,1]$, 
by considering a suitable change of variable. Then, by introducing special 
Jacobi parameters, the integral part is approximated using the Gauss--Jacobi 
quadrature rule. An approximation of the unknown function is considered in terms 
of Jacobi wavelets functions with unknown coefficients, which must be determined. 
By substituting this approximation into the equation, and collocating the 
resulting equation at a set of collocation points, a system of linear algebraic 
equations is obtained. Then, we suggest a method to determine the number of basis 
functions necessary to attain a certain precision. Finally, some examples 
are included to illustrate the applicability, efficiency, 
and accuracy of the new scheme.


\end{abstract}

% ------------------------------------------

\section{Introduction}
\label{sec:1}

In this paper, we consider the following Volterra integral equation (VIE)
\begin{equation}
\label{1.1}
t^\beta u(t)=f(t)+\int_0^t(t-x)^{-\alpha}\kappa(t,x)u(x)dx,
\quad t\in [0,T],
\end{equation}
where $\beta>0$, $\alpha\in[0,1)$, $\alpha+\beta\geq 1$, $f(t)=t^\beta g(t)$ 
with a continuous function $g$, and $\kappa$ is continuous on the domain 
$\Delta:=\{(t,x):0\leq x\leq t\leq T\}$ and is of the form
\begin{equation*}
\kappa(t,x)=x^{\alpha+\beta-1}\kappa_1(t,x),
\end{equation*}
where $\kappa_1$ is continuous on $\Delta$.
The existence of the term $t^{\beta}$ in the left-hand side of equation \eqref{1.1} 
gives it special properties, which are distinct from those of VIEs of the second kind 
(where the left-hand side is always different from zero), and also distinct 
from those of the first kind (where the left-hand side is constant and equal to zero). 
This is why in the literature they are often mentioned as VIEs of the third kind.
This class of equations has attracted the attention of researchers in the last years.
The existence, uniqueness, and regularity of solutions to \eqref{1.1} were discussed 
in \cite{Sonia1}. In that paper, the authors have derived necessary conditions to convert 
the equation into a cordial VIE, a class of VIEs which was studied in detail in
\cite{Vainikko1,Vainikko2}. This made possible to apply to equation \eqref{1.1}
some results known for cordial equations. In particular, the case $\alpha +\beta>1$
is of special interest, because in this case, if $\kappa_1(t,x) >0$, the integral 
operator associated to \eqref{1.1} is not compact and it is not possible to assure 
the solvability of the equation by classical numerical methods. In \cite{Sonia1}, 
the authors have introduced a modified graded mesh and proved that with such mesh 
the collocation method is applicable and has the same convergence order 
as for regular equations. 

In \cite{Nemati}, two of the authors of the present paper have applied 
a different approach, which consisted in expanding the solution as a series of
adjusted hat functions and approximating the integral operator by an operational matrix.
The advantage of that approach is that it reduces the  problem to a system of linear
equations with a simple structure, which splits into subsystems of three equations,
making the method efficient and easy to implement \cite{Nemati}. 
A limitation of this technique is that the optimal convergence 
order ($O(h^4)$) can be attained only under 
the condition that the solution satisfies $u\in C^4([0,T])$,
which is not the case in many applications.  

It is worth remarking here that there is a close connection between equations 
of class \eqref{1.1} and fractional differential equations \cite{MyID:410}. Actually, 
the kernel of \eqref{1.1} has the same form as the one of a fractional differential equation, 
and if we consider the case $\kappa(t,x)\equiv 1$, then the integral operator is 
the Riemann--Liouville operator of order $1-\alpha$. Therefore, it makes sense 
to apply to this class of equations numerical approaches that have recently been 
applied with success to fractional differential equations and related problems \cite{MyID:410}.

One of these techniques were wavelets, a set of functions built by dilation 
and translation of a single function $\phi(t)$, which is called the mother wavelet. 
These functions are known as a very powerful computational tool. 
The term wavelet was introduced by Jean Morlet about 1975 and the theory 
of the wavelet transform was developed by him and Alex Grossmann in the 80s \cite{Grossmann,Grossmann1}.
Some developments exist concerning the multiresolution analysis algorithm based on
wavelets \cite{Daubechies2} and the construction of compactly supported orthonormal 
wavelet basis \cite{Mallat}. Wavelets form an unconditional (Riesz) basis 
for $L^2(R)$, in the sense that any function in $L^2(R)$ can be decomposed 
and reconstructed in terms of wavelets \cite{Daubechies3}. 
Many authors have constructed and used different types of wavelets, 
such as B-spline \cite{XLi}, Haar \cite{Chen}, Chebyshev \cite{YLi}, 
Legendre \cite{Jafari}, and Bernoulli \cite{Rahimkani} wavelets. The advantage
of employing wavelets, in comparison with other basis functions, is that when using
them one can improve the accuracy of the approximation in two different ways:
(i) by increasing the degree of the mother function (assuming that it is a polynomial);  
(ii) by increasing the level of resolution, that is, reducing the support 
of each basis function.

We underline that the application of wavelets has special advantages 
in the case of equations with non-smooth solutions, as it is the case 
of equation \eqref{1.1}. In such cases, increasing the degree of the 
polynomial approximation is not a way to improve the accuracy 
of the approximation; however, such improvement can be obtained 
by increasing the level of resolution. 

In a recent work \cite{Nemati2}, Legendre wavelets were applied 
to the numerical solution of fractional delay-type integro-differential equations. 
In the present paper we will apply a close technique to approximate the solution of \eqref{1.1}.
  
The paper is organized as follows. Section~\ref{sec:2} is devoted to the 
required preliminaries for presenting the numerical technique. 
In Section~\ref{sec:3}, we give some error bounds for the best approximation 
of a given function by a generalized Jacobi wavelet. Section~\ref{sec:4} 
is concerned with the  presentation of a new numerical method for solving 
equations of type \eqref{1.1}. In Section~\ref{sec:5}, we suggest 
a criterion to determine the number of basis functions. Numerical examples 
are considered in Section~\ref{sec:6} to confirm the high accuracy 
and efficiency of this new numerical technique. Finally, 
concluding remarks are given in Section~\ref{sec:7}.  

% ----------------------------------

\section{Preliminaries}
\label{sec:2}

In this section, we present some definitions and basic concepts 
that will be used in the sequel.

% -------

\subsection{Jacobi wavelets}

The Jacobi polynomials $\left\{P_i^{(\nu,\gamma)}(t)\right\}_{i=0}^{\infty}$, 
$\nu,\gamma>-1$, $t\in[-1,1]$, are a set of orthogonal 
functions with respect to the weight function 
\begin{equation*}
w^{(\nu,\gamma)}(t)=(1-t)^\nu(1+t)^\gamma,
\end{equation*}
with the following orthogonality property:
\begin{equation*}
\int_{-1}^{1}w^{(\nu,\gamma)}(t)
P_i^{(\nu,\gamma)}(t)P_j^{(\nu,\gamma)}(t)dt=h_i^{(\nu,\gamma)}\delta_{ij},
\end{equation*}
where $\delta_{ij}$ is the Kronecker delta and 
\begin{equation*}
h_i^{(\nu,\gamma)}=\frac{2^{\nu+\gamma+1}\Gamma(\nu+i+1)
\Gamma(\gamma+i+1)}{i!(\nu+\gamma+2i+1)\Gamma(\nu+\gamma+i+1)}.
\end{equation*}
The Jacobi polynomials include a variety of orthogonal polynomials 
by considering different admissible values for the Jacobi parameters 
$\nu$ and $\gamma$. The most popular cases are the Legendre polynomials, 
which correspond to $\nu=\gamma=0$, Chebyshev polynomials of the first-kind,  
which correspond to $\nu=\gamma=-0.5$, and Chebyshev polynomials of the second-kind,  
which correspond to $\nu=\gamma=0.5$. 

We define the generalized Jacobi wavelets functions on the interval $[0,T)$ as follows:
\begin{equation*}
\psi_{n,m}^{(\nu,\gamma)}(t)
=\left\{
\begin{array}{ll}
2^{\frac{k}{2}}\sqrt{\frac{1}{h_m^{(\nu,\gamma)}T}}
P_m^{(\nu,\gamma)}\left(\frac{2^k}{T}t-2n+1\right),
& \frac{n-1}{2^{k-1}}T\leq t<\frac{n}{2^{k-1}}T,\\
0,&\text{otherwise},
\end{array}
\right.
\end{equation*}
where $k = 1, 2, 3, \ldots $ is the level of resolution, 
$n = 1, 2, 3,\ldots, 2^{k-1}$, $m = 0,1, 2, \ldots$, 
is the degree of the Jacobi polynomial, and $t$ is the normalized time. 
The interested reader can refer to \cite{Boggess,Walnut} for more details 
on wavelets. Jacobi wavelet functions are orthonormal 
with respect to the weight function
\begin{equation*}
w_k^{(\nu,\gamma)}(t)
=\left\{
\begin{array}{ll}
w_{1,k}^{(\nu,\gamma)}(t),& 0\leq t < \frac{1}{2^{k-1}}T,\\
w_{2,k}^{(\nu,\gamma)}(t),& \frac{1}{2^{k-1}}T\leq t < \frac{2}{2^{k-1}}T,\\
\quad \vdots &\\
w_{2^{k-1},k}^{(\nu,\gamma)}(t),& \frac{2^{k-1}-1}{2^{k-1}}T\leq t < T,\\
\end{array}
\right.
\end{equation*}
where
\begin{equation*}
w_{n,k}^{(\nu,\gamma)}(t)=w^{(\nu,\gamma)}\left(\frac{2^k}{T}t-2n+1\right),
\quad n=1,2,\ldots, 2^{k-1}.
\end{equation*}

An arbitrary function $u\in L^2[0,T)$ may be approximated 
using Jacobi wavelet functions as
\begin{equation*}
u(t)\simeq \Psi_{k,M}^{(\nu,\gamma)}(t) 
=\sum_{n=1}^{2^{k-1}}\sum_{m=0}^{M}u_{n,m}\psi_{n,m}^{(\nu,\gamma)}(t),
\end{equation*}
where
\begin{equation*}
\begin{split}
u_{n,m}&=\langle u(t),\psi_{n,m}^{(\nu,\gamma)}(t)\rangle_{w_{k}^{(\nu,\gamma)}}\\
&=\int_0^T w_k^{(\nu,\gamma)}(t)u(t)\psi_{n,m}^{(\nu,\gamma)}(t)dt\\
&=\int_{\frac{n-1}{2^{k-1}}T}^{\frac{n}{2^{k-1}}T}
w_{n,k}^{(\nu,\gamma)}(t)u(t)\psi_{n,m}^{(\nu,\gamma)}(t)dt.
\end{split}
\end{equation*}

% -------

\subsection{Gauss--Jacobi quadrature rule}

For a given function $u$, the Gauss--Jacobi quadrature formula is given by
\begin{equation*}
\int_{-1}^{1}(1-t)^{\nu}(1+t)^{\gamma}u(t)dt
=\sum_{l=1}^{N}\omega_lu(t_l)+R_N(u),
\end{equation*}
where $t_l$, $l=1,\ldots,N$, are the roots of $P_N^{(\nu,\lambda)}$,  
$\omega_l$, $l=1,\ldots,N$, are the corresponding weights 
given by (see \cite{Shen}):
\begin{equation}
\label{2.4}
\omega_l=\frac{2^{\nu+\gamma+1}\Gamma(\nu+N+1)\Gamma(\gamma+N+1)}{N!
\Gamma(\nu+\gamma+N+1)(\frac{d}{dt}P_N^{(\nu,\gamma)}(t_l))^2(1-t_l^2)},
\end{equation} 
and $R_N(u)$ is the remainder term which is as follows:
\begin{equation}
\label{2.5}
\begin{split}
R_N(u)=& \frac{2^{\nu+\gamma+2N+1}N!\Gamma(\nu+N+1)\Gamma(\gamma+N+1)
\Gamma(\nu+\gamma+N+1)}{(\nu+\gamma+2N+1)(\Gamma(\nu+\gamma+2N+1))^2}\\
&\quad\quad\quad\quad\quad\quad\quad\quad\quad\quad\quad\quad\quad\quad
\times\frac{u^{(2N)}(\eta)}{(2N)!},\quad \eta\in(-1,1).
\end{split}
\end{equation}
According to the remainder term \eqref{2.5}, the Gauss--Jacobi quadrature rule 
is exact for all polynomials of degree less than or equal to $2N-1$. 
This rule is valid if $u$ possesses no singularity in $(-1,1)$.  
It should be noted that the roots and weights of the Gauss--Jacobi 
quadrature rule can be obtained using numerical 
algorithms (see, e.g., \cite{Pang}).

% ------------------------------------------------

\section{Best approximation errors}
\label{sec:3}

The aim of this section is to give some estimates for the error 
of the Jacobi wavelets approximation of a function $u$ 
in terms of  Sobolev norms and seminorms. With this purpose, 
we extend to the case of Jacobi wavelets some results which were 
obtained in \cite{Canuto} for the best approximation error by Jacobi polynomials
in Sobolev spaces. The main result of this section is Theorem~\ref{thm:3.1}, 
which establishes a relationship between the regularity of a given function 
and the convergence rate of its approximation by Jacobi wavelets.

We first introduce some notations that will be used in this paper.  
Suppose that $L^2_{w^*}(a,b)$ is the space of measurable functions 
whose square is Lebesgue integrable in $(a,b)$ relative 
to the weight function $w^*$. The inner product and
norm of $L^2_{w^*}(a,b)$ are, respectively, defined by
\begin{equation*}
\langle u,v\rangle_{w^*}
=\int_a^b w^*(t)u(t)v(t)dt,
\quad \forall~u,v\in L^2_{w^*}(a,b),
\end{equation*}
and
\begin{equation*}
\left\| u\right\|_{L_{w^*}^2(a,b)}
=\sqrt{\langle u,u\rangle_{w^*}}.
\end{equation*}
The Sobolev norm of integer order $r\geq 0$ in the interval $(a,b)$, is given by
\begin{equation}
\label{4.1}
\parallel u \parallel _{ H_{w^*}^{r}(a,b)}
=\left({\sum^r _{j=0}\parallel u^{(j)}
\parallel^2 _{L_{w^*}^2(a,b)}} \right)^\frac{1}{2},
\end{equation}
where $u^{(j)}$ denotes the $j$th derivative of $u$ and $H_{w^*}^{r}(a,b)$ 
is a weighted Sobolev space relative to the weight function $w^*$. 

For ease of use, for some fixed values of $-1<\nu,\gamma<1$, we set 
\begin{equation*}
\psi_{i,j}(t):=\psi_{i,j}^{(\nu,\gamma)}(t), 
~ w(t):=w^{(\nu,\gamma)}(t), 
~ w_k(t):=w_k^{(\nu,\gamma)}(t),
~ w_{n,k}(t):=w_{n,k}^{(\nu,\gamma)}(t).
\end{equation*}
For starting the error discussion, first, we recall the following 
lemma from \cite{Canuto}.

\begin{lemma}[See \cite{Canuto}] 
\label{lem1}
Assume that $u\in H_w^\mu(-1,1)$ with $\mu\geq 0$ and $L_{M}(u) \in \mathbb{P}_{M}$ 
denotes the truncated Jacobi series of $u$, where $\mathbb{P}_{M}$ is the space 
of all polynomials of degree less than or equal to $M$. Then,
\begin{equation}
\label{4.2}
\parallel{ u-L_{M}(u)} \parallel_{L_w^2(-1,1)}
\leq CM^{-\mu}|u|_{H_w^{\mu;M}(-1,1)},
\end{equation}
where
\begin{equation}
\label{4.3}
\mid u \mid _{ H_w^{\mu;M}(-1,1)}
=\left({\sum^\mu _{j=\min\{\mu,M+1\}}\parallel 
u^{(j)}\parallel^2_{L_w^2(-1,1)}} \right)^\frac{1}{2}
\end{equation}
and $C$ is a positive constant independent of the function 
$u$ and integer $M$. Also, for $1\leq r \leq \mu$, one has
\begin{equation}
\label{4.4}
\parallel{ u-L_{M}(u)} \parallel_{H_w^r(-1,1)}
\leq CM^{2r-\frac{1}{2}-\mu}|u|_{H_w^{\mu;M}(-1,1)}.
\end{equation}
\end{lemma}

Suppose that $\Pi_{M}(I_{k,n})$ denotes the set of all 
functions whose restriction on each subinterval 
$I_{k,n}=\left(\frac{n-1}{2^{k-1}}T, \frac{n}{2^{k-1}}T\right)$, 
$n=1,2,\ldots,{2^{k-1}}$, are polynomials of degree 
at most $M$. Then, the following lemma holds.

\begin{lemma} 
\label{lem2} 
Let $u_{n} : I_{k,n} \rightarrow \mathbb{R}$, 
$n = 1, 2,\ldots ,2^{k-1}$, be functions in 
$H_{w_{n,k}}^\mu(I_{k,n})$ with $\mu\geq 0$. 
Consider the function $\bar{u}_{n}: (-1,1) \rightarrow  \mathbb{R}$ defined by
$(\bar{u}_{n})(t)=u_{n}\left(\frac{T}{2^{k}}(t+2n-1)\right)$ 
for all $t \in (-1,1)$. Then, for $0\leq j\leq \mu$, we have
\begin{equation*}
\parallel (\bar{u}_{n})^{(j)} \parallel _{ L_w^2(-1,1)}
=\left({\frac{2^{k}}{T}}\right)^{\frac{1}{2}-j}\parallel 
u^{(j)}_{n}\parallel _{L_{w_{n,k}}^2(I_{k,n})}.
\end{equation*}
\end{lemma}

\begin{proof}
Using the definition of the $L^2$-norm and the change 
of variable $t'=\frac{T}{2^{k}}(t+2n-1)$, we have
\begin{equation*}
\begin{split}
\parallel{{{{\bar {u}_{n}}}^{(j)}}} \parallel_{{L_w^2}( - 1,1)}^2 
&= \int_{ - 1}^1 w(t) |{\bar{u}_{n}}^{(j)} (t)|^2dt \\
&=\int_{ - 1}^1 w(t)\left|{u_{n}}^{(j)} \left(\frac{T}{2^{k}}(t+2n-1)\right)\right|^2dt \\
&= \int_{\frac{{n - 1}}{2^{k-1}}T}^{\frac{n}{2^{k-1}}T} 
w_{n,k}(t^{'})\left(\frac{2^{k}}{T}\right)^{-2j}\left|u_{n}^{(j)}(t')\right|^2\left(
\frac{2^{k}}{T}\right)dt'\\
&= \left(\frac{2^{k}}{T}\right)^{1 - 2j}\parallel 
{u_{n}^{(j)}}\parallel_{{L_{w_{n,k}}^2}(I_{k,n})}^2,
\end{split}
\end{equation*}
which proves the lemma.
\end{proof}

In order to continue the discussion, for convenience, we introduce 
the following seminorm for $u \in H_{w_k}^\mu(0,T)$, 
$0\leq r\leq \mu$, $M\geq 0$ and $k\geq 1$,
which replaces the seminorm \eqref{4.3} in the case 
of a wavelet approximation:
\begin{equation}
\label{4.6}
\mid u \mid _{ H_{w_k}^{r;\mu;M;k}(0,T)}
=\left({\sum^\mu _{j=\min\{\mu,M+1\}}
\left(2^{k}\right)^{2r-2j} \parallel u^{(j)}
\parallel ^2_{L_{w_k}^2(0,T)}} \right)^\frac{1}{2}.
\end{equation}
If we choose $M$ such that $M\geq \mu-1$, it can be easily seen that
\begin{equation}
\label{4.7}
\mid u \mid _{ H_{w_k}^{r;\mu;M;k}(0,T)}
=\left(2^{k}\right)^{r-\mu} \parallel u^{(\mu)}
\parallel _{L_{w_k}^2(0,T)}.
\end{equation}
The next theorem provides an estimate of the best approximation error, 
when Jacobi wavelets are used, in terms of the seminorm defined by \eqref{4.6}. 
\begin{theorem} 
\label{thm:3.1}
Suppose that $u \in H_{w_k}^\mu(0,T)$ with $\mu \geq 0$ and 
$$\Psi_{k,M}(u)=\sum\limits_{n = 1}^{{2^{k - 1}}}{\sum\limits_{m = 0}^{M} 
{{u_{n,m}}{\psi _{n,m}}(t)} },$$
is the best approximation 
of $u$ based on the Jacobi wavelets. Then,
\begin{equation}
\label{4.8}
\parallel{u - \Psi_{k,M}(u)} \parallel_{{L_{w_k}^2}(0,T)}
\leq CM^{-\mu}|u|_{H_{w_k}^{0;\mu;M;k}(0,T)}
\end{equation}
and, for $1\leq r \leq \mu$,
\begin{equation}
\label{4.9}
\parallel{u - \Psi_{k,M}(u)} \parallel_{{H_{w_k}^r}(0,T)}
\leq CM^{2r-\frac{1}{2}-\mu}|u|_{H_{w_k}^{r;\mu;M;k}(0,T)},
\end{equation}
where in \eqref{4.8} and \eqref{4.9} the constant $C$ denotes 
a positive constant that is independent of $M$ and $k$ 
but depends on the length $T$.
\end{theorem}
\begin{proof}
Consider the function $u_{n}:I_{k,n}\rightarrow \mathbb{R}$
such that $u_{n}(t)=u(t)$ for all $t\in I_{k,n}$. Then, 
from \eqref{4.1} and Lemma~\ref{lem2} for $0\leq r\leq \mu$, 
we have 
\begin{equation}
\label{4.10}
\begin{split}
\left\| {{u} - {\Psi _{k,M}}({u})} \right\|_{H_{w_k}^r(0,T )}^2 
&= \sum\limits_{n = 1}^{{2^{k - 1}}} {\left\| {{u_{n}} 
- \sum\limits_{m=0}^{M } {{u_{n,m}}{\psi _{n,m}}(t)} } 
\right\|} _{H_{w_{n,k}}^r({I_{k,n}})}^2\\ 
&= \sum\limits_{n = 1}^{2^{k - 1}}\sum_{j=0}^{r}\left\|{u^{(j)}_{n}} 
- \left(\sum\limits_{m=0}^{M } {{u_{n,m}}{\psi_{n,m}}(t)}\right)^{(j)}  
\right\|_{L_{w_{n,k}}^2(I_{k,n} )}\\
&=\sum\limits_{n = 1}^{2^{k - 1}} \sum\limits_{j = 0}^r 
{\left(\frac{2^{k}}{T}\right)}^{2j - 1}{\left\| {\bar u}_{n}^{(j)} 
-{\left(L_M({\bar u}_{n})\right)}^{(j)} \right\|}_{L_w^2( - 1,1)}^2\\
&\leq C_1\sum\limits_{n = 1}^{2^{k - 1}} \sum\limits_{j = 0}^r 
{\left(2^{k}\right)}^{2j - 1}{\left\| {\bar u}_{n}^{(j)} 
-{\left(L_M({\bar u}_{n})\right)}^{(j)} \right\|}_{L_w^2(- 1,1)}^2.
\end{split}
\end{equation}
By setting $r=0$ in \eqref{4.10}, we obtain
\begin{equation*}
\begin{split}
\left\|u - \Psi _{k,M }(u) \right\|_{L_{w_k}^2(0,T )}^2 
&\leq C_1\sum\limits_{n = 1}^{2^{k - 1}} 
{\left(2^{k}\right)}^{-1}{\left\| {\bar u}_{n} 
-{\left(L_M({\bar u}_{n})\right)}\right\|}_{L_w^2(-1,1)}^2\\
&\leq C_2M^{-2\mu}{\left(2^{k}\right)}^{ - 1}\sum\limits_{n = 1}^{2^{k - 1}}
\sum^\mu _{j=\min\{\mu,M+1\}}{\left\| {\bar u}_{n}^{(j)}\right\|}_{L_w^2( - 1,1)}^2\\
&\leq CM^{-2\mu}\sum^\mu _{j=\min\{\mu,M+1\}}{\left(2^{k}\right)}^{- 2j}
\sum\limits_{n = 1}^{2^{k - 1}}{\left\| { u}_{n}^{(j)}\right\|}_{L_{w_{n,k}}^2( I_{n,k})}^2\\
&= CM^{-2\mu}\sum^\mu _{j=\min\{\mu,M+1\}}{\left(2^{k}\right)}^{- 2j}{\left\| 
{u}^{(j)}\right\|}_{L_{w_{k}}^2(0,T)}^2,
\end{split}
\end{equation*}
where we have used \eqref{4.2} and Lemma~\ref{lem2}. This completes 
the proof of \eqref{4.8}. Furthermore, using \eqref{4.10} 
for $1\leq r \leq \mu$ and $k\geq 1$, we get:
\begin{align*}
\left\| {{u} - {\Psi _{k,M}}({u})} \right\|_{H_{w_k}^r(0,T )}^2 
&\leq C_1{\left(2^{k}\right)}^{2r - 1} \sum\limits_{n = 1}^{2^{k - 1}} 
\sum\limits_{j = 0}^r {\left\| {\bar u}_{n}^{(j)} 
-{\left(L_M({\bar u}_{n})\right)}^{(j)} \right\|}_{L_w^2( - 1,1)}^2\\
&=C_1{\left(2^{k}\right)}^{2r - 1} \sum\limits_{n = 1}^{2^{k - 1}}{
\left\| {\bar u}_{n} -{L_M({\bar u}_{n})} \right\|}_{H^r_w( - 1,1)}^2\\
&\leq C_2M^{4r-1-2\mu}{\left(2^{k}\right)}^{2r - 1}\sum\limits_{n = 1}^{2^{k - 1}}
\sum^\mu_{j=\min\{\mu,M+1\}}{\left\| {\bar u}_{n}^{(j)}\right\|}_{L_w^2( - 1,1)}^2\\
&=C_2M^{4r-1-2\mu}{\left(2^{k}\right)}^{2r - 1}\sum^\mu _{j=\min\{\mu,M+1\}}
\sum\limits_{n = 1}^{2^{k - 1}}{\left\| {\bar u}_{n}^{(j)}\right\|}_{L_w^2( - 1,1)}^2\\
&\leq CM^{4r-1-2\mu}\sum^\mu _{j=\min\{\mu,M+1\}}\left(2^k\right)^{2r-2j}
\sum\limits_{n = 1}^{2^{k - 1}}\left\| u_{n}^{(j)}\right\|^2_{L^2_{w_{n,k}(I_{k,n})}}\\
&= CM^{4r-1-2\mu}\sum^\mu _{j=\min\{\mu,M+1\}}\left(2^k\right)^{2r-2j}\left\| 
u^{(j)}\right\|^2_{L^2_{w_{k}(0,T)}},
\end{align*}
where we have used \eqref{4.1}, \eqref{4.4}, and Lemma~\ref{lem2}. 
Therefore, we have proved \eqref{4.9}.
\end{proof}

\begin{remark}
We can also obtain estimates for the Jacobi wavelets approximation in terms
of the $L^2$-norm.  With $M\geq \mu-1$,
if we combine \eqref{4.7} with \eqref{4.8}, we get
\begin{equation*}
\parallel{u - \Psi_{k,M}(u)} \parallel_{{L_{w_k}^2}(0,T)}
\leq CM^{-\mu}2^{-\mu k} \parallel u^{(\mu)}
\parallel _{L_{w_k}^2(0,T)};
\end{equation*}
from \eqref{4.7} and\eqref{4.8}, we obtain
\begin{equation*}
\parallel{u - \Psi_{k,M}(u)} \parallel_{{H_{w_k}^r}(0,T)}
\leq CM^{2r-\frac{1}{2}-\mu}\left(2^{k}\right)^{r-\mu} \parallel u^{(\mu)}
\parallel _{L_{w_k}^2(0,T)},\quad r\geq 1.
\end{equation*}
\end{remark} 

% ------------------------------------

\section{Method of solution}
\label{sec:4}

In this section, we propose a method for solving the VIE \eqref{1.1}. 
To this end, by using a suitable change of variable, we transform the 
interval of the integral to $[-1,1]$. Suppose that 
\begin{equation*}
s=2\left(\frac{x}{t}\right)-1,\quad ds=\frac{2}{t}dx.
\end{equation*}
Therefore, the equation \eqref{1.1} is transformed 
into the following integral equation:
\begin{equation}
\label{3.1}
t^\beta u(t)=f(t)+{\left(\frac{t}{2}\right)}^{1-\alpha}
\int_{-1}^{1}(1-s)^{-\alpha}\kappa(t,\frac{t}{2}(s+1))
u(\frac{t}{2}(s+1))ds.
\end{equation}
In order to compute the integral part of \eqref{3.1}, 
we set $\nu=-\alpha$ and $\gamma=0$ as the Jacobi parameters 
and use the Gauss--Jacobi quadrature rule. Then, we have
\begin{equation}
\label{3.2}
t^\beta u(t)=f(t)+{\left(\frac{t}{2}\right)}^{1-\alpha}\sum_{l=1}^N
\omega_l\kappa(t,\frac{t}{2}(s_l+1))u(\frac{t}{2}(s_l+1)),
\end{equation}
where $s_l$ are the zeros of $P_N^{(-\alpha,0)}$ 
and $w_l$ are given using \eqref{2.4} as
\begin{equation*}
\omega_l=\frac{2^{1-\alpha}}{\left(\frac{d}{dx}
P_N^{(-\alpha,0)}(s_l)\right)^2\left(1-s_l^2\right)},
\quad l=1,2,\ldots,N.
\end{equation*}

We consider an approximation of the solution of \eqref{1.1} 
in terms of the Jacobi wavelets functions as follows:
\begin{equation}
\label{3.3}
u(t)\simeq\sum_{i=1}^{2^{k-1}}\sum_{j=0}^{M}
u_{i,j}\psi_{i,j}^{(\nu,\gamma)}(t),
\end{equation}
where the Jacobi wavelets coefficients $u_{i,j}$ are unknown. 
In order to determine these unknown coefficients, 
we substitute \eqref{3.3} into \eqref{3.2} and get
\begin{equation}
\label{3.4}
\begin{split}
\sum_{i=1}^{2^{k-1}}\sum_{j=0}^{M}&\left[t^\beta \psi_{i,j}^{(\nu,\gamma)}(t)
-{\left(\frac{t}{2}\right)}^{1-\alpha}\sum_{l=1}^N\omega_l
\kappa\left(t,\frac{t}{2}(s_l+1)\right)\right.\\
&\quad\quad\quad\quad\quad\quad\quad\quad\quad
\quad\quad\quad\quad\quad\left.\times\psi_{i,j}^{(\nu,\gamma)}\left(
\frac{t}{2}(s_l+1)\right)\right]u_{i,j}=f(t).
\end{split}
\end{equation}
In this step, we define the following collocation points
\begin{equation*}
t_{n,m}=\frac{T}{2^k}\left( \tau_m+2n-1\right),
\quad n=1,2,\ldots,2^{k-1},~m=0,1,\ldots,M,
\end{equation*}
where $\tau_m$, $m=0,1,\ldots,M$, are the zeros of $P_{M+1}^{(\nu,\gamma)}$. 
Therefore, $t_{n,m}$ are the shifted Gauss--Jacobi points in the interval
$\left(\frac{n-1}{2^{k-1}}T,\frac{n}{2^{k-1}}T\right)$, 
corresponding to the Jacobi parameters $\nu$ and $\gamma$. By collocating 
\eqref{3.4} at the points $t_{n,m}$, we obtain
\begin{equation}
\label{3.5}
\begin{split}
\sum_{i=1}^{2^{k-1}}\sum_{j=0}^{M}\left[t_{n,m}^\beta
\psi_{i,j}^{(\nu,\gamma)}(t_{n,m})-\right.&{\left(\frac{t_{n,m}}{2}\right)}^{1-\alpha}
\sum_{l=1}^N\omega_l\kappa\left(t,\frac{t_{n,m}}{2}(s_l+1)\right)\\
&\quad\quad\quad\left.\times \psi_{i,j}^{(\nu,\gamma)}\left(
\frac{t_{n,m}}{2}(s_l+1)\right)\right]u_{i,j}=f(t_{n,m}).
\end{split}
\end{equation}
By considering $n=1,2,\ldots,2^{k-1}$ and $m=0,1,\ldots,M$, in the above equation, 
we have a system of linear algebraic equations that can be rewritten 
as the following matrix form:
\begin{equation}
\label{3.6}
AU=F,
\end{equation}
where
\begin{equation*}
U=\left[\begin{array}{c}
u_{1,0}\\
u_{1,1}\\
\vdots\\
u_{1,M}\\
\vdots\\
u_{2^{k-1},0}\\
u_{2^{k-1},1}\\
\vdots\\
u_{2^{k-1},M}\\
\end{array}
\right],\quad
\quad 
F=\left[\begin{array}{c}
f(t_{1,0})\\
f(t_{1,1})\\
\vdots\\
f(t_{1,M})\\
\vdots\\
f(t_{2^{k-1},0})\\
f(t_{2^{k-1},1})\\
\vdots\\
f(t_{2^{k-1},M})\\
\end{array}
\right], 
\end{equation*}
and the entries of the rows of the matrix $A$ are the expressions in the bracket 
in \eqref{3.5}, which vary corresponding to the values of $i$ and $j$, i.e., 
the coefficients of $u_{i,j}$, $i=1,2,\ldots,2^{k-1}$, $j=0,1,\ldots,M$, 
for $t_{n,m}$, that are all nonzero and positive. Since the functions 
$\psi_{i,j}^{(\nu,\gamma)}$ are orthonormal and the nodes $t_{n,m}$ 
are pairwise distinct, the matrix $A$ is nonsingular. Therefore, \eqref{3.6} 
is unique solvable. By solving this system using a direct method, 
the unknown coefficients $u_{i,j}$ are obtained. Finally, an approximation 
of the solution of \eqref{1.1} is given by \eqref{3.3}. 

% ------------------------------------

\section{A criterion for choosing the number of wavelets}
\label{sec:5}

Now we discuss the choice of adequate values of $k$ and $M$ (number of basis functions). 
To do this, we suppose that $u(\cdot)\in C^{2N}([0,T])$ 
and $\kappa(\cdot,\cdot)\in C^{2N}([0,T]\times [0,T])$. Using the error of the 
Gauss--Jacobi quadrature rule given by \eqref{2.5}, and substituting 
$\nu=-\alpha$ and $\gamma=0$ in it, the exact solution 
of \eqref{1.1} satisfies the equation 
\begin{equation*}
t^\beta u(t)=f(t)+{\left(\frac{t}{2}\right)}^{1-\alpha}\left[\sum_{l=1}^N
\omega_l\kappa\left(t,\frac{t}{2}(s_l+1)\right)u\left(\frac{t}{2}(s_l+1)\right)
+R_N(\kappa u)\right],
\end{equation*}
where 
\begin{equation*}
\begin{split}
R_N(\kappa u)=&\frac{2^{-\alpha+2N+1}(N!)^2\left(
\Gamma(-\alpha+N+1)\right)^2}{(2N)!(-\alpha+2N+1)\left(
\Gamma(-\alpha+2N+1)\right)^2}\\
&\quad\quad\quad\quad\quad\quad\quad\quad\quad\quad\quad\quad
\times {\left(\frac{t}{2}\right)}^{2N}\frac{\partial^{2N}
\left(\kappa(t,x)u(x)\right)}{\partial x^{2N}}|_{x=\eta},
\end{split}
\end{equation*}
for $\eta\in(0,T)$. Therefore, we have
\begin{equation*}
\begin{split}
t^\beta u(t)=f(t)+{\left(\frac{t}{2}\right)}^{1-\alpha}
&\left(\sum_{l=1}^N\omega_l\kappa\left(t,\frac{t}{2}(s_l+1)\right)
u\left(\frac{t}{2}(s_l+1)\right)\right)\\
&\quad\quad\quad+t^{1-\alpha+2N}\xi_{\alpha,N}\frac{\partial^{2N}\left(
\kappa(t,x)u(x)\right)}{\partial x^{2N}}|_{x=\eta},
\end{split}
\end{equation*}
where
\begin{equation*}
\xi_{\alpha,N}=\frac{(N!)^2\left(\Gamma(-\alpha+N+1)\right)^2}{(2N)!
(-\alpha+2N+1)\left(\Gamma(-\alpha+2N+1)\right)^2}.
\end{equation*}
Suppose that $t\neq 0$. Then we obtain that
\begin{equation}
\label{7.1}
\begin{split}
u(t)=g(t)+2^{\alpha-1}{t}^{1-\alpha-\beta}
&\left(\sum_{l=1}^N\omega_l\kappa\left(t,\frac{t}{2}(s_l+1)\right)
u\left(\frac{t}{2}(s_l+1)\right)\right)\\
&\quad\quad\quad+t^{1-\alpha-\beta+2N}\xi_{\alpha,N}
\frac{\partial^{2N}\left(\kappa(t,x)u(x)\right)}{\partial x^{2N}}|_{x=\eta}.
\end{split}
\end{equation}
Let $U_{k,M}(t)=\sum\limits_{i=1}^{2^{k-1}}\sum\limits_{j=0}^{M}u_{i,j}
\psi_{i,j}^{(\nu,\gamma)}(t)$ be the numerical solution of \eqref{1.1} 
obtained by the proposed method in Section~\ref{sec:4}. From the definition 
of the Jacobi wavelets, for the collocation points $t_{n,m}$, 
$n=1,\ldots,2^{k-1}$, $m=0,1,\ldots,M$, we have 
\begin{equation*}
U_{k,M}(t_{n,m})=\sum_{j=0}^{M}u_{n,j}\psi_{n,j}^{(\nu,\gamma)}(t_{n,m}),
\quad n=1,\ldots,2^{k-1}.
\end{equation*}
By definition, the restriction of the functions $\psi_{n,j}^{(\nu,\gamma)}(t)$, 
$n=1,\ldots,2^{k-1}$, on the subinterval $I_{k,n}$, which we denote here 
by $\rho_{n,j}^{(\nu,\gamma)}(t)$, is smooth. Therefore, we can define
\begin{equation*}
\zeta_{n,j}=\max_{x,t\in I_{k,n}}\left|\frac{\partial^{2N}\left(
\kappa(t,x)\rho_{n,j}^{(\nu,\gamma)}(x)\right)}{\partial x^{2N}}\right|.
\end{equation*}
For a given $\varepsilon>0$, since all the collocation points, 
$t_{n,m}$, $n=1,\ldots,2^{k-1}$, $m=0,1,\ldots,M$, are positive, 
using \eqref{7.1}, we can choose $k$ and $M$ such that for all $t_{n,m}$, 
the following criterion holds:
\begin{multline*}
\left|U_{k,M}(t_{n,m})-g(t_{n,m})-2^{\alpha-1}t_{n,m}^{1-\alpha-\beta}\left(
\sum_{l=1}^N\omega_l\kappa\left(t_{n,m},\frac{t_{n,m}}{2}(s_l+1)\right)\right.\right.\\
\left.\left.\times U_{k,M}\left(\frac{t_{n,m}}{2}(s_l+1)\right)\right)\right|
+t_{n,m}^{1-\alpha-\beta+2N}\xi_{\alpha,N}\left|
\sum_{j=0}^Mu_{n,j}\zeta_{n,j}\right|<\varepsilon.
\end{multline*}

% ------------------------------------

\section{Numerical examples}
\label{sec:6}

In this section, we consider three examples of VIEs of the third-kind 
and apply the proposed method to them. The weighted $L^2$-norm is used 
to show the accuracy of the method. In all the examples, we have used 
$N=10$ in the Gauss--Jacobi quadrature formula and the following 
notation is used to show the convergence of the method:
\begin{equation*}
\text{Ratio}=\frac{e(k-1)}{e(k)},
\end{equation*}
where $e(k$) is the $L^2$-error obtained with resolution $k$.

% -------------

\begin{example}
\label{ex1} 
As the first example, we consider the following third-kind VIE, 
which is an equation of Abel type \cite{Sonia1,Nemati}:
\begin{equation*}
t^{2/3}u(t)=f(t)+\int_0^t \frac{\sqrt{3}}{3\pi}
x^{1/3}(t-x)^{-2/3}u(x)dx, \quad t\in[0,1],
\end{equation*}
where
\begin{equation*}
f(t)=t^{\frac{47}{12}}\left(1-\frac{\Gamma(\frac{1}{3})
\Gamma(\frac{55}{12})}{\pi\sqrt{3}\Gamma(\frac{59}{12})}\right).
\end{equation*}
The exact solution of this equation is $u(t)=t^{\frac{13}{4}}$, 
which belongs to the space $H^3([0,1])$. We have employed the method 
for this example with different values of $M$, $k$, $\nu$ and 
$\gamma$, and reported the results in 
Tables~\ref{tab:1}, \ref{tab:4} and Figure~\ref{fig:1}. 
Table~\ref{tab:1} displays the weighted $L^2$-norm of the error 
for three different classes of the Jacobi parameters, which include 
$\nu=\gamma=0.5$ (second-kind Chebyshev wavelets), 
$\nu=\gamma=0$ (Legendre wavelets), 
and $\nu=\gamma=-0.5$ (first-kind Chebyshev wavelets) 
with different values of $M$ and $k$. Moreover, the ratio 
of the error versus $k$ is given in this table. 
It can be seen from Table~\ref{tab:1} that the method converges 
faster in the case of the second-kind Chebyshev wavelets. 
In Table~\ref{tab:4}, we compare the maximum absolute error at the 
collocation points obtained by our method with the error of the 
collocation method introduced in \cite{Sonia1}, and the operational 
matrix method based on the adjusted hat functions \cite{Nemati}. 
From this table, it can be seen that our method gives more accurate 
results with less collocation points (we have used $192$) than the method of 
\cite{Sonia1} ($N=256$) and also, has higher accuracy with 
a smaller number of basis functions (we have used $192$) than the method 
of \cite{Nemati} ($193$). In Figure~\ref{fig:1}, we show the error 
function obtained by the method based on the second-kind Chebyshev 
wavelets with $M=6$, $k=4$ (left) and $M=6$, $k=6$ (right).   
% -----------------------------
\begin{table}[!ht]
\centering
\caption{(Example \ref{ex1}.) Numerical results 
with different values of $M$ and $k$.}\label{tab:1}
\begin{tabular}{lllllllll}
\hline
&\multicolumn{8}{c}{$\nu=\gamma=0.5$}\\
\cline{2-9}
&\multicolumn{2}{c}{$M=3$} && \multicolumn{2}{c}{$M=4$}&&\multicolumn{2}{c}{$M=5$}\\
\cline{2-3}\cline{5-6}\cline{8-9}
$k$ & $L^2$-Error & Ratio && $L^2$-Error & Ratio &&$L^2$-Error & Ratio\\
\hline
$1$ &$6.61e-4$ & --- &  & $6.16e-5$ & ---   &&$1.19e-5$ & --- \\
$2$ &$5.68e-5$ & $11.64$&& $4.47e-6$ &$13.78$ &&$8.20e-7$ & $14.51$\\
$3$ &$4.26e-6$ & $13.33$&& $2.92e-7$ &$15.31$ &&$5.28e-8$ & $15.53$\\
$4$ &$3.06e-7$ & $13.92$&& $1.86e-8$ &$15.70$ &&$3.34e-9$ & $15.81$\\
$5$ &$2.15e-8$ &$14.23$ && $1.18e-9$&$15.76$ &&$2.12e-10$&$15.75$\\
\hline
&\multicolumn{8}{c}{$\nu=\gamma=0$}\\
\cline{2-9}
&\multicolumn{2}{c}{$M=3$} && \multicolumn{2}{c}{$M=4$}&&\multicolumn{2}{c}{$M=5$}\\
\cline{2-3}\cline{5-6}\cline{8-9}
$k$ & $L^2$-Error & Ratio && $L^2$-Error & Ratio &&$L^2$-Error & Ratio\\
\hline
$1$ &$8.93e-4$ & --- &  & $8.46e-5$ & ---   &&$1.66e-5$ & --- \\
$2$ &$7.12e-5$ & $12.54$&& $6.33e-6$ &$13.36$ &&$1.24e-6$ & $13.39$\\
$3$ &$5.54e-6$ & $12.85$&& $4.71e-7$ &$13.44$ &&$9.19e-8$ & $13.49$\\
$4$ &$4.24e-7$ & $13.07$&& $3.50e-8$ &$13.46$ &&$6.83e-9$ & $13.46$\\
$5$ &$3.22e-8$ & $13.17$ && $2.60e-9$  &$13.46$ &&$5.09e-10$&$13.42$\\
\hline
&\multicolumn{8}{c}{$\nu=\gamma=-0.5$}\\
\cline{2-9}
&\multicolumn{2}{c}{$M=3$} && \multicolumn{2}{c}{$M=4$}&&\multicolumn{2}{c}{$M=5$}\\
\cline{2-3}\cline{5-6}\cline{8-9}
$k$ & $L^2$-Error & Ratio && $L^2$-Error & Ratio &&$L^2$-Error & Ratio\\
\hline
$1$ &$1.29e-3$ & --- &  & $1.24e-4$ & ---   &&$2.48e-5$ & --- \\
$2$ &$1.12e-4$ & $11.52$&& $1.04e-5$ &$11.92$ &&$2.12e-6$ & $11.70$\\
$3$ &$9.78e-6$ & $11.45$&& $9.12e-7$ &$11.40$ &&$1.85e-7$ & $11.46$\\
$4$ &$8.57e-7$ & $11.41$&& $8.02e-8$ &$11.37$ &&$1.63e-8$ & $11.35$\\
$5$ &$7.53e-8$ & $11.38$&& $7.07e-9$  &$11.34$ &&$1.44e-9$& $11.32$\\
\hline
\end{tabular}
\end{table}
% -----------------------------
\begin{table}[!ht]
\centering
\footnotesize
\caption{(Example \ref{ex1}.) Comparison 
of the maximum absolute error.}\label{tab:4}
\begin{tabular}{lllcccccc}
\hline
\multicolumn{3}{c}{Present method ($M=5$, $k=6$)}
&&\multicolumn{2}{c}{Method of \cite{Sonia1} (Radau II)}
&&\multicolumn{2}{c}{Method of \cite{Nemati}}\\
\cline{1-3}\cline{5-6}\cline{8-9}
$\nu=\gamma=0.5$ & $\nu=\gamma=0$ & $\nu=\gamma=0.5$ 
& &\multicolumn{2}{c}{$m=3$, $N=256$} 
&&\multicolumn{2}{c}{$n=192$}\\
\hline
$2.81e-10$ &$2.02e-10$ &$1.15e-10$ 
&  & \multicolumn{2}{c}{$5.13e-9$} & &\multicolumn{2}{c}{$5.16e-9$} \\
\hline
\end{tabular}
\end{table}
% -----------------------------
\begin{figure}[!ht]
\centering
\includegraphics[scale=0.7]{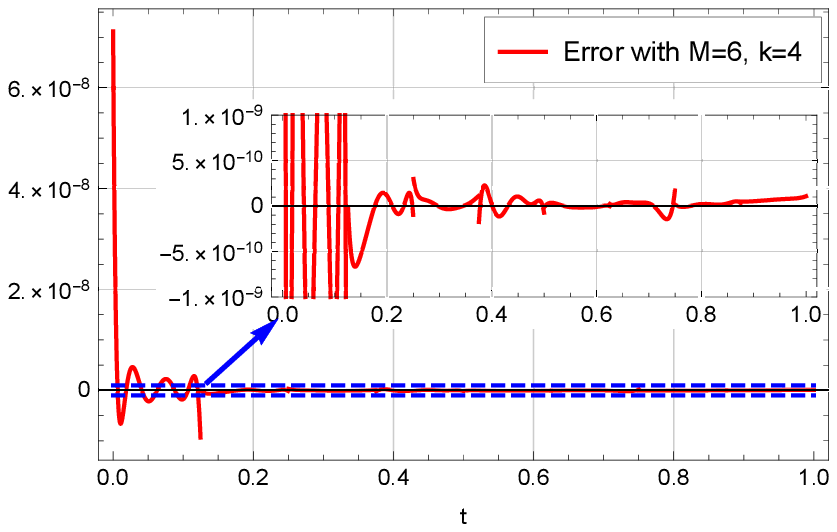}
\includegraphics[scale=0.7]{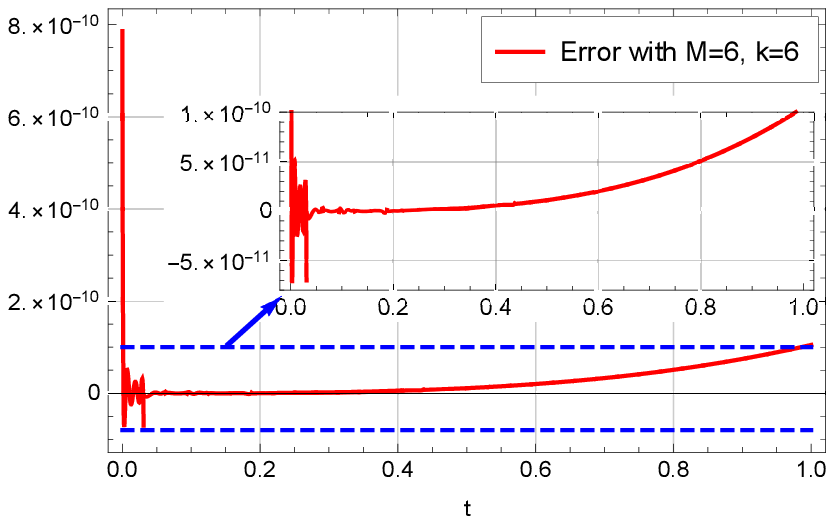}
\caption{(Example \ref{ex1}.) Plot of the error function 
with $\nu=\gamma=0.5$ and $M=6$, $k=4$ (left), and $M=6$, $k=6$ (right).}\label{fig:1}
\end{figure}
\end{example}

% -------------

\begin{example}
\label{ex2} 
Consider the following third-kind VIE, which is used in the modelling 
of some heat conduction problems with mixed-type boundary 
conditions \cite{Sonia1,Nemati}:
\begin{equation*}
tu(t)=\frac{6}{7}t^3\sqrt{t}+\int_0^t \frac{1}{2}u(x)dx,
\quad t\in[0,1].
\end{equation*}
This equation has the exact solution $u(t)=t^{\frac{5}{2}}$ ($u\in H^2([0,1])$). 
The numerical results are given in Tables~\ref{tab:2}, \ref{tab:5} and Figure~\ref{fig:2}. 
The $L^2$-norms and the ratio of the error given in Table~\ref{tab:2} 
confirm the superiority of the second-kind Chebyshev wavelets compared to the 
Legendre wavelets and first-kind Chebyshev wavelets. The method 
converges slower in this example than in the previous one, which 
could be expected, due to the lower regularity of the solution.
A comparison between the maximum absolute error at collocation points 
of the present method and the methods given in \cite{Sonia1} and \cite{Nemati} 
is presented in Table~\ref{tab:5}. Moreover, the error functions in the case 
of the second-kind Chebyshev wavelets with $M=6$, $k=4$ and $M=6$, $k=6$, 
can be seen in Figure~\ref{fig:2}.  
% -----------------------------
\begin{table}[!ht]
\centering
\caption{(Example~\ref{ex2}.) Numerical results 
with different values of $M$ and $k$.}\label{tab:2}
\begin{tabular}{lllllllll}
\hline
&\multicolumn{8}{c}{$\nu=\gamma=0.5$}\\
\cline{2-9}
&\multicolumn{2}{c}{$M=3$} && \multicolumn{2}{c}{$M=4$}&&\multicolumn{2}{c}{$M=5$}\\
\cline{2-3}\cline{5-6}\cline{8-9}
$k$ & $L^2$-Error & Ratio && $L^2$-Error & Ratio &&$L^2$-Error & Ratio\\
\hline
$1$ &$1.06e-3$ & --- &  & $2.42e-4$ & ---   &&$7.83e-5$ & --- \\
$2$ &$1.38e-4$ & $7.68$&& $3.00e-5$ &$8.07$ &&$9.92e-6$ & $7.89$\\
$3$ &$1.74e-5$ & $7.93$&& $3.69e-6$ &$8.13$ &&$1.23e-6$ & $8.07$\\
$4$ &$2.15e-6$ & $8.09$&& $4.56e-7$ &$8.09$ &&$1.52e-7$ & $8.09$\\
$5$ &$2.64e-7$ & $8.14$ && $5.59e-8$ &$8.16$ &&$1.93e-8$ & $7.88$\\
\hline
&\multicolumn{8}{c}{$\nu=\gamma=0$}\\
\cline{2-9}
&\multicolumn{2}{c}{$M=3$} && \multicolumn{2}{c}{$M=4$}&&\multicolumn{2}{c}{$M=5$}\\
\cline{2-3}\cline{5-6}\cline{8-9}
$k$ & $L^2$-Error & Ratio && $L^2$-Error & Ratio &&$L^2$-Error & Ratio\\
\hline
$1$ &$1.07e-3$ & --- &  & $2.33e-4$ & ---   &&$7.33e-5$ & --- \\
$2$ &$1.38e-4$ & $7.75$&& $3.01e-5$ &$7.74$ &&$9.45e-6$ & $7.76$\\
$3$ &$1.79e-5$ & $7.71$&& $3.86e-6$ &$7.80$ &&$1.21e-6$ & $7.81$\\
$4$ &$2.29e-6$ & $7.82$&& $4.95e-7$ &$7.80$ &&$1.55e-7$ & $7.81$\\
$5$ &$2.93e-7$ & $7.82$ && $6.32e-8$  &$7.83$ &&$2.04e-8$&$7.60$\\
\hline
&\multicolumn{8}{c}{$\nu=\gamma=-0.5$}\\
\cline{2-9}
&\multicolumn{2}{c}{$M=3$} && \multicolumn{2}{c}{$M=4$}&&\multicolumn{2}{c}{$M=5$}\\
\cline{2-3}\cline{5-6}\cline{8-9}
$k$ & $L^2$-Error & Ratio && $L^2$-Error & Ratio &&$L^2$-Error & Ratio\\
\hline
$1$ &$1.28e-3$ & --- &   & $2.78e-4$ & ---   &&$8.82e-5$ & --- \\
$2$ &$1.84e-4$ & $6.96$&& $4.04e-5$ &$6.88$ &&$1.29e-5$ & $6.84$\\
$3$ &$2.71e-5$ & $6.79$&& $5.97e-6$ &$6.77$ &&$1.91e-6$ & $6.75$\\
$4$ &$4.02e-6$ & $6.74$&& $8.86e-7$ &$6.74$ &&$2.83e-7$ & $6.75$\\
$5$ &$5.97e-7$ & $6.73$&& $1.32e-7$ &$6.71$ &&$4.23e-8$& $6.69$\\
\hline
\end{tabular}
\end{table}
% -----------------------------
\begin{table}[!ht]
\centering
\footnotesize
\caption{(Example~\ref{ex2}.) Comparison of the maximum absolute error.}\label{tab:5}
\begin{tabular}{lllcccccc}
\hline
\multicolumn{3}{c}{Present method ($M=5$, $k=6$)}
&&\multicolumn{2}{c}{Method of \cite{Sonia1} (Chebyshev)}
&&\multicolumn{2}{c}{Method of \cite{Nemati}}\\
\cline{1-3}
$\nu=\gamma=0.5$ & $\nu=\gamma=0$ & $\nu=\gamma=0.5$ 
& &\multicolumn{2}{c}{$m=2$, $N=256$} &&\multicolumn{2}{c}{$n=192$}\\
\hline
$3.70e-8$ &$2.69e-8$ &$1.46e-8$ 
&  & \multicolumn{2}{c}{$1.46e-8$} 
& &\multicolumn{2}{c}{$3.46e-8$} \\
\hline
\end{tabular}
\end{table}
% -----------------------------
\begin{figure}[!ht]
\centering
\includegraphics[scale=0.7]{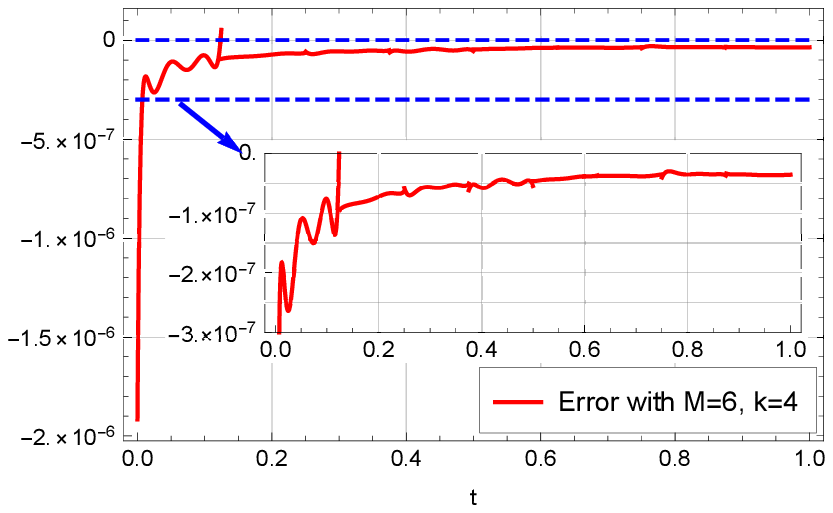}
\includegraphics[scale=0.7]{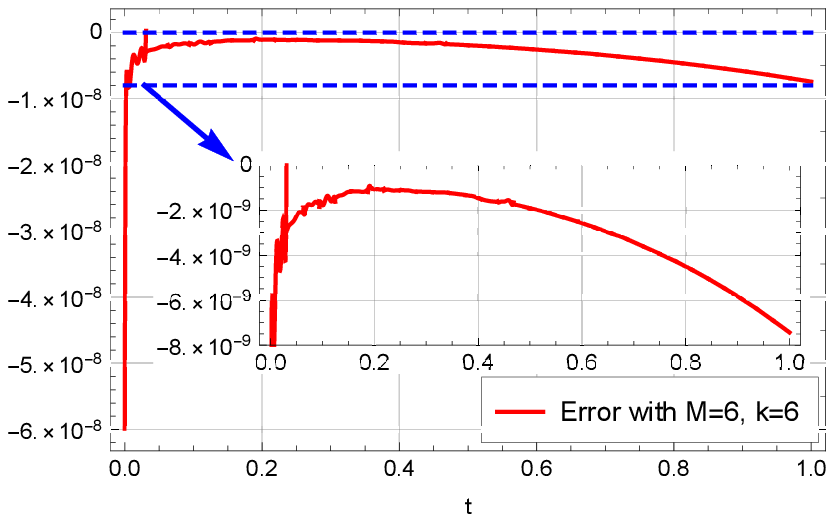}
\caption{(Example~\ref{ex2}) Plot of the error function with 
$\nu=\gamma=0.5$ and $M=6$, $k=4$ (left), 
and $M=6$, $k=6$ (right).}\label{fig:2}
\end{figure}
\end{example}

% -------------

\begin{example}
\label{ex3} 
Consider the following VIE of the third kind:
\begin{equation*}
t^{3/2}u(t)=f(t)+\int_0^t \frac{\sqrt{2}}{2 \pi } x (t-x)^{-1/2 }u(x)dx,
\quad x\in[0,1],
\end{equation*}
where
\begin{equation*}
f(t)=t^{33/10} \left(1-\frac{\Gamma \left(\frac{19}{5}\right)}{\sqrt{2 \pi} 
\Gamma \left(\frac{43}{10}\right)}\right).
\end{equation*}
This equation has the exact solution $u(t)=t^{\frac{9}{5}}$ ($u\in H^1([0,1])$). 
The numerical results for this example are displayed in Table~\ref{tab:3} 
and Figure~\ref{fig:3}, which confirm the higher accuracy of the second-kind 
Chebyshev wavelets method compared with the Legendre wavelets and first-kind 
Chebyshev wavelets methods. Since the exact solution in this case is not 
so smooth as in the previous examples, the method converges slower.
% --------------------
\begin{table}[!ht]
\centering
\caption{(Example~\ref{ex3}.) Numerical results 
with different values of $M$ and $k$.}\label{tab:3}
\begin{tabular}{lllllllll}
\hline
&\multicolumn{8}{c}{$\nu=\gamma=0.5$}\\
\cline{2-9}
&\multicolumn{2}{c}{$M=3$} && \multicolumn{2}{c}{$M=4$}&&\multicolumn{2}{c}{$M=5$}\\
\cline{2-3}\cline{5-6}\cline{8-9}
$k$ & $L^2$-Error & Ratio && $L^2$-Error & Ratio &&$L^2$-Error & Ratio\\
\hline
$1$ &$4.18e-4$ & --- &  & $1.35e-4$ & ---   &&$5.48e-5$ & --- \\
$2$ &$7.88e-5$ & $5.30$&& $2.46e-5$ &$5.49$ &&$9.80e-6$ & $5.59$\\
$3$ &$1.39e-5$ & $5.67$&& $4.27e-6$ &$5.76$ &&$1.70e-6$ & $5.76$\\
$4$ &$2.40e-6$ & $5.79$&& $7.36e-7$ &$5.80$ &&$2.92e-7$ & $5.82$\\
$5$ &$4.12e-7$ & $5.83$&& $1.26e-7$ &$5.84$ &&$4.99e-8$ & $5.85$\\
\hline
&\multicolumn{8}{c}{$\nu=\gamma=0$}\\
\cline{2-9}
&\multicolumn{2}{c}{$M=3$} && \multicolumn{2}{c}{$M=4$}&&\multicolumn{2}{c}{$M=5$}\\
\cline{2-3}\cline{5-6}\cline{8-9}
$k$ & $L^2$-Error & Ratio && $L^2$-Error & Ratio &&$L^2$-Error & Ratio\\
\hline
$1$ &$6.13e-4$ & --- &  & $2.06e-4$ & ---   &&$8.73e-5$ & --- \\
$2$ &$1.25e-4$ & $4.90$&& $4.18e-5$ &$4.93$ &&$1.77e-5$ & $4.93$\\
$3$ &$2.53e-5$ & $4.94$&& $8.50e-6$ &$4.92$ &&$3.60e-6$ & $4.92$\\
$4$ &$5.14e-6$ & $4.92$&& $1.73e-6$ &$4.91$ &&$7.31e-7$ & $4.92$\\
$5$ &$1.04e-6$ & $4.94$&& $3.51e-7$ &$4.93$ &&$1.48e-7$ & $4.94$\\
\hline
&\multicolumn{8}{c}{$\nu=\gamma=-0.5$}\\
\cline{2-9}
&\multicolumn{2}{c}{$M=3$} && \multicolumn{2}{c}{$M=4$}&&\multicolumn{2}{c}{$M=5$}\\
\cline{2-3}\cline{5-6}\cline{8-9}
$k$ & $L^2$-Error & Ratio && $L^2$-Error & Ratio &&$L^2$-Error & Ratio\\
\hline
$1$ &$9.70e-4$ & --- &  & $3.40e-4$ & ---   &&$1.50e-4$ & --- \\
$2$ &$2.27e-4$ & $4.27$&& $8.04e-5$ &$4.23$ &&$3.57e-5$ & $4.20$\\
$3$ &$5.43e-5$ & $4.18$&& $1.93e-5$ &$4.17$ &&$8.59e-6$ & $4.16$\\
$4$ &$1.31e-5$ & $4.15$&& $4.65e-6$ &$4.15$ &&$2.07e-6$ & $4.15$\\
$5$ &$3.15e-6$ & $4.16$&& $1.12e-6$ &$4.15$ &&$4.99e-7$ & $4.15$\\
\hline
\end{tabular}
\end{table}
% -----------------------------
\begin{figure}[!ht]
\centering
\includegraphics[scale=0.7]{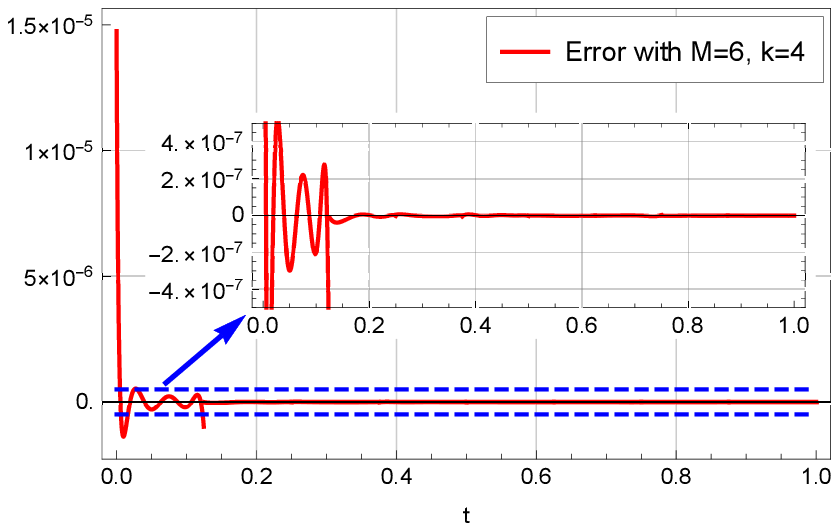}
\includegraphics[scale=0.7]{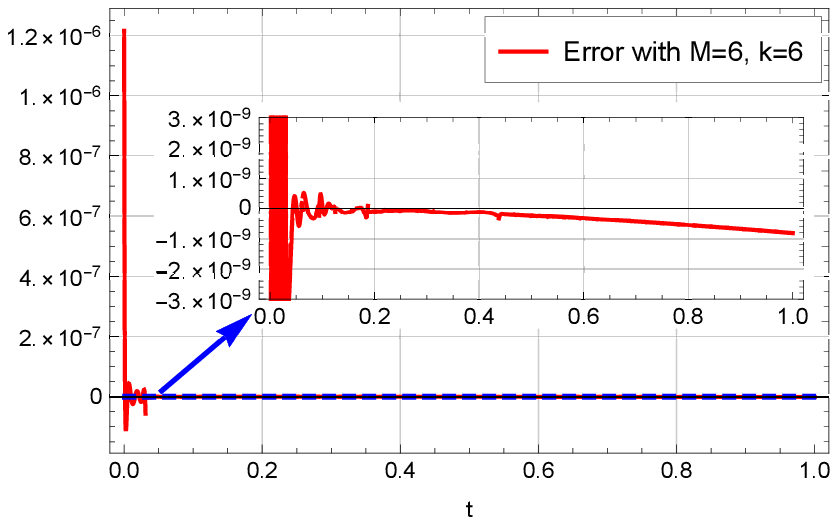}
\caption{(Example~\ref{ex3}.) Plot of the error function 
with $\nu=\gamma=0.5$ and $M=6$, $k=4$ (left), 
and $M=6$, $k=6$ (right).}\label{fig:3}
\end{figure}
\end{example}

% ------------------------------------

\section{Concluding remarks}
\label{sec:7}

In this work, a numerical method based on Jacobi wavelets 
has been introduced for solving a class of Volterra integral 
equations of the third-kind. First, the Jacobi wavelet functions 
have been introduced. Some error bounds are presented for the 
best approximation of a given function using the Jacobi wavelets. 
A numerical method based on the Jacobi wavelets, together with the 
use of the Gauss--Jacobi quadrature formula, has been proposed 
in order to solve Volterra integral equations of the third-kind. 
A criterion has been introduced for choosing the number of basis 
functions necessary to reach a specified accuracy. Numerical results 
have been included to show the applicability and high accuracy 
of this new technique. Our results confirm that the new method 
has higher accuracy than the other existing methods 
to solve the considered class of equations.  

% ------------------------------------

\section*{Acknowledgment}

Pedro M. Lima acknowledges support from Funda\c c\~ao para 
a Ci\^encia e a Tecnologia (FCT, the Portuguese Foundation 
for Science and Technology) through the grant 
UID/MAT/04621/2019 (CEMAT); Delfim F. M. Torres was supported 
by FCT within project UIDB/04106/2020 (CIDMA).

% ------------------------------------

% ------------------------------------

\end{document}